\documentclass{article}

\usepackage{amsmath,amssymb,amsfonts}
\usepackage{amsthm}
\usepackage{xspace,enumerate}
\usepackage{burdges}

\newcommand\oo{^\circ}

\def\margin#1{}

\begin{document}

\title{Weyl groups of small groups\\ of finite Morley rank} 
\author{%
Jeffrey Burdges\thanks{Supported by NSF postdoctoral fellowship DMS-0503036.}
\ and
Adrien Deloro}
\maketitle

\begin{abstract}
We examine Weyl groups of minimal connected simple groups of finite Morley rank of degenerate type. We show that they are cyclic, and lift isomorphically to subgroups of the ambient group.
\end{abstract}

\section{Introduction}

Infinite groups of finite Morley rank have little truly geometric
structure; however, their algebraic properties are remarkably
reminiscent of algebraic groups.  The strongest conjecture to this
effect is the {\em Cherlin-Zilber algebraicity conjecture} which
posits that an infinite simple group of finite Morley rank is
a linear algebraic group over an algebraically closed field.

There are a number of partial results towards this conjecture and
a complete proof in the {\em even \& mixed type} cases \cite{ABC_EvenType},
 read potentially characteristic two.
In other cases, much recent work has followed three themes :
 various identification theorems \cite{Bu_balance,Deloro_PSL2},  
 an analysis of the minimal simple groups
  where Bender's method is well understood \cite{Bu_Bender}, and
 the analysis of torsion using genericity arguments \cite{BBC,BC_semisimple},
 as well as work involving several techniques \cite{CJ01,BCJ,Deloro_NotPSL2,AB_CT}.

A major part of the combined thread is the analysis of
 the Weyl group $W = N(T)/C(T)$ of $G$ 
 associated to some maximal decent torus $T$.
Here a {\em decent torus} is merely the smallest definable subgroup
 containing some divisible abelian torsion subgroup,
 such as a $p$-torus $\Z(p^\infty)^n$.
One may speak of the Weyl group of $G$ because
 maximal decent tori are conjugate in a group of finite Morley rank \cite{Ch05}.

The main result of the present article is that the Weyl group of
 a minimal connected simple group of finite Morley rank is cyclic.
We also show that, for $p > 2$ dividing $|W|$,
 the group $G$ has no divisible $p$-torsion.

\begin{namedtheorem}{Theorem \ref{thm:Weyl_cyclic}}
Let $G$ be a minimal connected simple group of finite Morley rank,
 and let $T$ be a decent torus of $G$.
Then the Weyl group $W := N(T)/C(T)$ is cyclic,
 and has an isomorphic lifting to $G$;
 moreover no primes dividing $|W|$ appear in $T$ except possibly 2.
\end{namedtheorem}

\noindent In spirit, the proof proceeds by proving the three conclusions
 in the reverse order.

Our primary concern here is the case of groups of {\em degenerate type},
 where Sylow 2-subgroups are finite, or equivalently trivial \cite{BBC}.
The second author's thesis work covers the {\em odd type} case
 (see Fact \ref{DJ_odd_cyclic} below),
 where the Sylow 2-subgroup is divisible-abelian-by-finite,
and the afore mentioned classification in even \& mixed types \cite{ABC_EvenType}
 covers those cases.

In degenerate type, one may also extract some unlikely number theoretic
 consequences from this argument.

\begin{namedtheorem}{Corollary \ref{BW}*}
Let $G$ be a minimal connected simple group of finite Morley rank
 and degenerate type, and let $T$ be a nontrivial decent torus of $G$.
Then $G$ interprets a bad field in characteristic $p$
 for every prime divisor $p$ of $|W|$.
\end{namedtheorem}

A bad field $(k,H,+,\cdot)$ is a field $k$ of finite Morley rank
 with a proper nontrivial definable subgroup $H$ of it's multiplicative subgroup.
Such fields exist in characteristic zero \cite{BadField}
 but are quite unlikely in positive characteristic $p$ because then :
\begin{enumerate}
\item 
 there are only finitely many primes of
 the form $p^n-1 \over p-1$ by \cite{Wag}.
\item
$(p^n-1)_\pi \asymp p^{\alpha n}$ where 
 $\pi$ denotes the set of primes appearing
 in the bad field's multiplicative subgroup $H$,
 $(\cdot)_\pi$ denotes removing all primes but these,
 and $\alpha := \rk(H)/\rk(k)$ \cite{HruWag_Newton}.
\end{enumerate}
\noindent These conditions are regarded as unlikely by specialists.

As an exercise, we point out that Corollary \ref{BW} is quite easy
 when $p$ is the {\em minimal} prime divisor of $|W|$
 and $T$ is a maximal decent torus. 
First, Fact \ref{f:Weyl_p} provides a Borel subgroup $B$
 such that $U_p(B) \neq 1$. 
So one considers a $B$-minimal $A \leq U_p(B)$.
If $C_B(A) < B$ then Zilber's field theorem produces the bad field, as desired.
Otherwise $C_B(A) = B$ implies that $\cup B^G$ is generic in $G$.
So then, by \cite[Theorem 1]{BC_semisimple}, 
 $B$ contains a maximal decent torus of $G$, whose Weyl group normalizes $B$.
Finally the Weyl group element may not live either inside or outside $B$
 by Facts \ref{selfnormsolv} and \ref{unipotentselfnormalization}.

In \S\ref{sec:Literature} we expose minor adjustments of
 various known results, and recall some facts that will be  used.
\S\ref{sec:Cartan} is devoted to analyzing Cartan subgroups $C(T)$
 where $T$ is a maximal decent torus, as well as
 how Weyl cosets lift into elements of the group.
Theorem \ref{thm:Weyl_cyclic} is proved in \S\ref{sec:Cyclicity}. In \S\ref{sec:Counting_Borel} we conclude that under the assumption that maximal decent tori are large enough and there is a non-trivial Weyl group, only finitely many Borel subgroups can contain the Carter subgroup.

\subsection*{Acknowledgments}

The authors thank Tuna \Altinel, Alexandre Borovik, and Gregory Cherlin
for their helpful advice and corrections.  The authors also gratefully
acknowledge the hospitality and support of Universit\'e Paris 7,
University of Manchester, the Camerino Modnet meeting, and the Cherlin family.

\section{Literature}\label{sec:Literature}

This section serves as general background. However, the reader should be aware that later sections freely use standard material from \cite{BN} without any comment.

\subsection{Weyl Groups}

The Weyl group of an algebraic group is the quotient of the normalizer of a maximal algebraic torus by the centralizer of the same torus.
In general, groups of finite Morley rank need not contain any algebraic torus,
 but frequently have so-called decent tori.
A {\em decent torus} is a divisible abelian group which is the
 {\em definable hull} of its torsion,
i.e.\ is the intersection of all definable subgroups containing its torsion subgroup.
The group $W_T := N(T)/C(T)$ is the Weyl group
 associated to a given torus $T$.
Here one may use only the connected component $C^\o (T)$ of the centralizer
 by the following.

\begin{fact}[{\cite[Theorem 1]{AB_CT}}]\label{CTconnected}
Let $T$ be a decent torus of a connected group $H$ of finite Morley rank.
Then $C(T)$ is connected.
\end{fact}

We naturally say ``the Weyl group \emph{of $G$}'' when $T$ is a maximal decent torus;
this is well-defined because
 maximal decent tori are always conjugate \cite[Last lines]{Ch05}.

In the present article, we use repeatedly the fact that
 connected solvable groups of finite Morley rank have trivial Weyl groups.
This fact is originally due to Oliver \Frecon,
 but we employ the following formulation.

\begin{fact}[{\cite[Lemma 6.6]{AB_CT}}]\label{selfnormsolv}
Let $H$ be a connected solvable group of finite Morley rank and
let $K$ be a definable connected subgroup of $H$
 such that $[N_H(K):K]<\infty$.
Then $N_H(K) = K$.
\end{fact}

A priori, if $T$ is a decent subtorus of another torus $S$, then its Weyl group $W_T$ is just a section of $W_S$. But in the special case of minimal connected simple groups, one can say more thanks to the preceding fact.

\begin{lemma}\label{lem:Weyl_increase}
Let $G$ be a minimal connected simple group of finite Morley rank,
 and let $T$ be a nontrivial decent torus of $G$.
Then the Weyl group $W_T = N(T)/C(T)$ associated to $T$
 is naturally isomorphic to
 a subgroup of the Weyl group $W$ of $G$.
\end{lemma}

\begin{proof}
Let $S$ be a maximal decent torus containing $T$, so that $W = W_S = N(S)/C(S)$.
Then $S \leq C(T) \normal N(T)$.
A Frattini argument using conjugacy of maximal decent tori
 says that $N(T) = C(T) \cdot N_{N(T)} (S)$ and
$$ W_T = N(T)/C(T) = N_{N(T)} (S)/ N_{C(T)} (S) \mathperiod $$
Of course $C(T)$ is solvable by minimal simplicity of $G$
 and connected by Fact \ref{CTconnected}.
So Fact \ref{selfnormsolv} says that $N_{C(T)} (S) = C_{C(T)} (S) = C(S)$, 
which implies that
\[ W_T = N_{N(T)}(S) / C(S) \leq N(S) / C(S) = W_S = W. \qedhere \]
\end{proof}

A similar argument shows that, if $G$ is a minimal connected simple group
 containing a non-trivial decent torus,
 the Weyl group of $G$ is isomorphic to $W_Q := N(Q)/Q$
 where $Q$ is a Carter subgroup containing the maximal decent torus $T$.
Indeed one need not specify the Carter subgroup $Q$
 as they are all conjugate \cite{Frecon_Carter_conjugacy}.

\subsection{Connected torsion}

We shall employ numerous results about connected torsion subgroups.
A connected solvable $p$-subgroup of a group of finite Morley rank
is always a central product of a $p$-torus,
 i.e.\ a power of the \Prufer $p$-group $\Z(p^\infty)$,
 and a $p$-unipotent subgroup,
 i.e.\ a definable connected nilpotent $p$-group of bounded exponent.
A number of recent results provide criteria either
 forcing unipotence to appear or preventing it from occurring.

\begin{fact}[{\cite[Theorem 3]{BC_semisimple}}]\label{f:torality}
Let $G$ be a connected group of finite Morley rank, $\pi$ a set of primes,
and $a$ any $\pi$-element of $G$ such that $C_G\oo(a)$ has $\pi^\perp$ type.
Then $a$ belongs to a $\pi$-torus.
\end{fact}

\begin{fact}[{\cite[Corollary 5.3]{BC_semisimple}}]\label{f:Weyl_p}
Let $G$ be a minimal connected simple group of finite Morley rank.
Suppose the Weyl group is nontrivial and has odd order,
 with $r$ the smallest prime divisor of its order.
Then $G$ contains a unipotent $r$-subgroup in the
 centralizer of any $r$-element representing an $r$-element of $W$.
\end{fact}

\begin{fact}[{\cite[Lemma 4.3]{AB_CT}}]\label{unipotentselfnormalization}
Let $G$ be a minimal connected simple group of finite Morley rank.
Let $B$ be a Borel subgroup of $G$ such that $U_p(B) \neq 1$
 for some prime $p$. Then $p\not|[N_G(B):B]$.
\end{fact}

The preceding facts can be used to prove that
 minimal connected simple groups are covered by their Borel subgroups.

\begin{fact}[{\cite[Corollary 4.4]{AB_CT}}]\label{BCx}
Let $G$ be a minimal connected simple group of finite Morley rank.
Any torsion element $x$ of $G$ lies inside
 any Borel subgroup of $G$ which contains $C^\o(x)$.
\end{fact}

We will make crucial use of the following striking consequence.

\begin{lemma}\label{lem:constant_order}
Let $G$ be a minimal connected simple group of finite Morley rank,
 let $T$ be a nontrivial maximal torus, and let $\bar{x} \in N(T)/C(T)$.
Then any lifting of $\bar{x}$ to a torsion element $x \in N(T) \setminus C(T)$
 has the same order as $\bar{x}$.
In other words $\gen{x} \cap C(T) = 1$ whenever $x \in N(T)\setminus C(T)$ is torsion.
\end{lemma}

\begin{proof}
Suppose towards a contradiction that $x^n \in C(T)$ with $x^n \neq 1$.
Fix some Borel subgroup $B$ containing $C^\o(x^n) \geq C^\o (x)$.
Since $x$ is torsion, $x \in B$ by Fact \ref{BCx}.
But $T \leq C^\o(x^n) \leq B$ too,
 contradicting Fact \ref{selfnormsolv}.
\end{proof}


\subsection{Intersections}

In general, $p$-unipotent subgroups are more difficult to handle than tori,
 but their role inside minimal connected simple groups is well understood,
 thanks to the following trivial alteration of \cite[Proposition 3.11]{CJ01}.

\begin{fact}[{\cite[Lemma 2.1]{Bu_Bender}}]\label{jaligot}\label{interpperp}
Let $G$ be a minimal connected simple group.
Let $B_1,B_2$ be two distinct Borel subgroups of $G$
 satisfying $U_{p_i}(B_i) \neq 1$ for some prime $p_i$ ($i=1,2$).
Then $F(B_1) \cap F(B_2) = 1$.
\end{fact}

Here the $p$-unipotent radical $U_p(\cdot)$ denotes
 the largest $p$-unipotent subgroup, and
the Fitting subgroup $F(\cdot)$ is the largest normal nilpotent subgroup.
One normally invokes this fact to say that,
 if $B$ is a Borel subgroup with $U_p(B) \neq 1$, then
 $B \cap B^g$ is $p^\perp$ for all $g \not \in N_G(B)$.
Otherwise one considers a third Borel subgroup
 containing $C(t)$ where $t \in B\cap B^g$ has order $p$.
Such a situation violates Fact \ref{jaligot} because
 $C_{U_p(B)}(t)$ and $C_{U_p(B^g)}(t)$ would be infinite.


%
%

%
%

As in \cite{FJ_Conjugacy}, a {\it maximal unipotence parameter} $\tilde{q}$ of
 a group $B$ of finite Morley rank is a pair $(p, r)$ where $p$ is a prime number or $\infty$,
 and $r$ an integer or $\infty$, such that $p = \infty \Leftrightarrow r < \infty$.

The phrase ``$\tilde{q}$-subgroup" has the obvious meaning if $p$ is prime;
 but ``$\tilde{q}$-subgroup" means a $U_{(0,r)}$-subgroup \cite{BuPhd} in the second case.

\begin{fact}[{\cite[Lemme 1.9.1]{DeloroPhd}}]\label{Marseilles}
Let $G$ be a minimal connected simple group,
 and let $B$ be a Borel subgroup of $G$.
Also let $\tilde{q}$ be a maximal unipotence parameter of $B$.
If $A$ is a normal $\tilde{q}$-subgroup of $B$, then
 $B$ is the only Borel subgroup containing $A$
 for which $\tilde{q}$ is still a maximal unipotence parameter.
\end{fact}

\subsection{Actions}

We require the following easy result which we have not located in the literature.

\begin{lemma}\label{CUpBt_inf}
Let $G$ be a group of finite Morley rank and $U \leq G$ be a $p$-unipotent subgroup.
If $t \in N_G(U)$ is such that $C_U(t) \neq 1$, then $C_U(t)$ is infinite.

In particular, if $G$ is a minimal connected simple group
 and $B \leq G$ a Borel subgroup with $U_p (B) \neq 1$,
then $C_{U_p(B)} (t)$ is infinite for any $t \in G$ with $C_{U_p(B)}(t) \neq 1$.
\end{lemma}

This will follow from the following two facts.

\begin{fact}[{\cite[Fact 3.3]{Bu_signalizer}}]\label{f:CsumC}
Let $H = K T$ be a group of finite Morley rank.
Suppose that $T$ is a solvable $\pi$-group of bounded exponent and
 that $K$ is a definable abelian normal $\pi^\perp$-subgroup of $H$.
Then $H = [H,T] \oplus C_H(T)$.
\end{fact}

\begin{fact}[{\cite[Lemma 2.5]{BC_generation}, see also \cite[Fact 3.12]{BuPhd}}]\label{f:C_v_mod_A}
Let $H$ be a solvable group of finite Morley rank,
 $v$ a definable automorphism of $H$ of order $q^n$ for some prime $q$, and
 $K$ a definable normal $v$-invariant connected $q^\perp$-subgroup of $H$.
Then $$ C_H\oo(v \bmod K) = C_H\oo(v) K/K \mathperiod $$
\end{fact}

The notation $C_G\oo(v \bmod K)$ refers to the connected component
 of the preimage in $G$ of $C_{G/K}(v)$, and $G^\# := G \setminus \{1\}$.

\begin{proof}[Proof of Lemma \ref{CUpBt_inf}]
Let $x \in C_U(t)^\#$ and let $Z_i := Z_i^\o( U )$ for $i \in \N$.
As $U$ is nilpotent and connected,
 there is some integer $i$ such that $x \in Z_{i+1} \setminus Z_i$.
In particular $t$ when acting on the \emph{connected} abelian quotient
 $Y := Z_{i+1}/Z_i$ centralizes $\bar{x}$.

We now write $t = uvw$ with $u, v, w \in d(t)$
 such that $d(u)$ is divisible, $v$ is a $p'$-element, and $w$ is a $p$-element.
By \cite[Corollary 9]{Wag01},
 there are no torsion-free definable sections of the multiplicative subgroup
 of a field of finite Morley rank in positive characteristic.
In consequence, $Y \cdot d(u)$ must be nilpotent, that is $u$ centralizes $Y$.

Notice that $v\in d(t) \leq C(\bar{x})$.
By Fact \ref{f:CsumC}, $C_Y(v) \cong Y/[Y,v]$ is connected.
So $C_Y(v)$ is nontrivial because it contains $\bar{x}$,
 and hence $Y_1 := C_Y(v) = C_Y (uv)$ is infinite.

Eventually the $p$-element $w$ acts on $Y_1$.
As $Y_1$ is an infinite $p$-group of bounded exponent,
 $Y_2 := C_{Y_1} (w) \leq C_Y (uvw)$ is infinite, by \cite[Corollary 6.20]{BN}.
Hence $C_Y(t)$ is infinite.
It follows from Fact \ref{f:C_v_mod_A} that $C_{Z_i}(t)$,
 and then also $C_U(t)$, are infinite.

To prove the second claim, it suffices to notice that 
 $t \in N_G (B) = N_G (U_p (B))$ by Fact \ref{jaligot}.
\end{proof}

We will also use the following generation principle.

\begin{fact}[{\cite[Fact 3.7]{Bu_signalizer}}]\label{Cgeneration}
Let $H$ be a solvable $q^\perp$-group of finite Morley rank.
Let $E$ be a finite elementary abelian $q$-group acting definably on $H$.
Then $H = \langle C_H(E_0) : E_0 \leq E, [E:E_0]=q \rangle$.
\end{fact}

Another useful fact is this nilpotence criterion.

\begin{fact}[{\cite[Theorem 2.4.7]{Wag} plus \cite[Exercise 14 p.79]{BN}}]\label{solvnilp}
Let $H$ be a connected solvable group of finite Morley rank
 with an automorphism of prime order whose centralizer in $H$ is finite.
Then $H$ is nilpotent.
\end{fact}



Almost any article about Weyl groups in our context will employ
 its $p$-adic representation :
$\End(\Z(p^\infty)^n)$ is the ring $M_{n \times n}(\Z_p)$ of
 $n \times n$ matrices over the $p$-adic integers $\Z_p$.
For us, this isomorphism, sometimes referred to as the Tate module,
 will be exploited via the following argument.

\begin{fact}\label{irrpol}
Let $G$ be a group of finite Morley rank.
Let $T$ be a $p$-torus and $a \in N(T) \setminus C(T)$ be a $p$-element.
Then the \Prufer $p$-rank of $T$ is at least $p-1$.
\end{fact}

\begin{proof}
We may assume that $a^p$ centralizes $T$.
Represent $a$ as an automorphism of $\GL_d (\Z_p)$, 
 where $d$ denotes the \Prufer $p$-rank of $T$, and work in $\GL_d (\Q_p)$.
The minimal polynomial $\mu_a$ divides $X^p - 1 = (X - 1) (1 + X + \dots + X^{p-1})$.
Of course $a \not \in C(T)$ implies $\mu_a \neq X - 1$.
On the other hand $1 + X + \dots + X^{p-1}$ is irreducible over $\Q_p$.
So it follows that $\mu_a $ has degree at least $p-1$.
Now the characteristic polynomial has degree $d$ and is a multiple of $\mu_a$,
 whence $d \geq p-1$.
\end{proof}

\section{Cartan subgroups}\label{sec:Cartan}

A Cartan subgroup of an algebraic group is the centralizer of a maximal torus.
In this section we analyze the ``Cartan subgroup'' $C(T)$ of our group
 by analyzing representatives of the Weyl group,
 eventually proving the following.

\begin{theorem}[Henri/\'Elie]\label{HE}
Let $G$ be a minimal connected simple group of degenerate type.
Suppose also that $G$ has a nontrivial Weyl group $W := N(T)/C(T)$
 where $T$ is a maximal decent torus.
Then the Cartan subgroup $C(T)$ is nilpotent,
 and thus is a Carter subgroup of $G$.

Moreover, $C(T)$ is actually a Borel subgroup
 if either $C(T)$ is not abelian or
 $G$ has \Prufer $q$-rank $\geq 3$ for some prime $q$.
\end{theorem}

Throughout this section we use the hypotheses and notation of Theorem \ref{HE},
 i.e.\ $G$ is a minimal connected simple group of degenerate type
 with a nontrivial Weyl group $W := N(T)/C(T)$.%
\margin{$G$,$T$,$W$}
Of course the decent torus $T$ is necessarily nontrivial under this hypothesis.

\subsection{Weylian elements}

We say that a torsion element $a$ of $G$ is {\em Weylian}
 if it normalizes some maximal decent torus $T$.
Also let $\tau$ denote the set of primes occurring in a maximal decent torus $T$.
By conjugacy of decent tori \cite{Ch05},
 $\tau$ is just the set of primes for which $G$ contains divisible torsion.%
\margin{Weylian,$\tau'$}
Then $$ \tau' :=
 \Setst{ \textrm{$p$ prime} }{ \textrm{$\Z(p^\infty)$ does not embed into $G$} }
 \mathperiod $$
Of course, an element $a$ of $G$ is $\tau'$
 iff its (finite) order $|a|$ is a $\tau'$ number.
We note that both of these properties are closed under taking powers.
In \S\ref{sec:Cyclicity}, we will show that
 any element of $W$ lifts to a Weylian $\tau'$-element.

\begin{lemma}\label{Ba}
If $a$ is a $\tau'$-element, then
 there is a unique Borel subgroup $B_a$ containing $C^\o(a)$,
 and $a$ is a product of unipotent elements of $B_a$.
\end{lemma}

\begin{proof}
Let $B_a$ be a Borel subgroup containing $C^\o(a)$.
Then $a$ lies inside $B_a$ by Fact \ref{BCx}.
As $G \geq B$ has no divisible $\tau'$ torsion,
 $a$ lies inside $\Pi_p U_p(B_a)$, proving the second part.
So the first part follows from Fact \ref{jaligot}.
\end{proof}

We preserve this $B_a$ notation for the unique Borel containing $C^\o(a)$,
 when $a$ is $\tau'$, throughout the remainder of the article.%
\margin{$B_a$}

\begin{lemma}\label{UpCT=1}
If $p \in \tau'$ divides $|W|$, then $C(T)$ is $p^\perp$.
\end{lemma}

We recall the usual torsion lifting lemma.

\begin{fact}[{\cite[Ex.~11 p.~93; Ex.~13c p.~72]{BN}}]\label{f:nodeftorsion}
Let $G$ be a group of finite Morley rank and let $H \normal G$ be a
definable subgroup.  If $x\in G$ is an element such that $\bar{x}\in G/H$
is a $p$-element for some prime $p$, then $x H$ contains a $p$-element.
Furthermore, if $H$ and $G/H$ are $p^\perp$-groups, then $G$ is a $p^\perp$-group.
\end{fact}

\begin{proof}[Proof of Lemma \ref{UpCT=1}]
There is a Weylian $p$-element $a$ normalizing $C(T)$ by Fact \ref{f:nodeftorsion}.
Clearly $U_p( C(T) ) \neq 1$ since $T$ is $\pi^\perp$ by definition
 (see \cite[\S6.4]{BN}).
If $U_p( C(T) ) \neq 1$,
 then $T \leq C(T) \leq B_a$ by Fact \ref{jaligot}.
As $a\in B_a$ too,
 this contradicts the fact that connected solvable groups
 have trivial Weyl groups (Fact \ref{selfnormsolv}).
\end{proof}

\begin{lemma}\label{a_not_piT}
If $a \in G$ is Weylian,
 and $U_p( C^\o(a) ) \neq 1$ for each prime divisor $p$ of $|a|$,
then $a$ is a $\tau'$-element.
\end{lemma}

\begin{proof}
Suppose that some prime $p \in \tau$ divides $|a|$.
Then there is a power $b$ of $a$ which is a $p$-element.
We may assume that $a$, and $b$, normalize $T$,
 by conjugacy of decent tori \cite{Ch05}.
By \cite[Corollary 6.20]{BN}, 
 there must be $t \in T^\#$ of order $p$ such that $[b, t] = 1$.

By assumption,
 the Borel subgroup $B$ containing $C^\o(b)$
 satisfies $U_p(B) \neq 1$. 
It follows from Fact \ref{unipotentselfnormalization} that $b \in B$.
As $b \in B \cap B^t$, $t \in N(B)$ by Fact \ref{jaligot}.
So again $t \in B$ using Fact \ref{unipotentselfnormalization}.

As $t$ has order $p$ too,
 it follows that $C_{U_p(B)}^\o(t) \neq 1$ by \cite[Corollary 6.20]{BN}. 
So again $T \leq C^\o(t) \leq B$ by Fact \ref{jaligot}.
Now once more we contradict the fact that connected solvable groups
 have trivial Weyl groups (Fact \ref{selfnormsolv}).

So there is no $p$-torus in $T$, or hence in $G$, as desired.
\end{proof}

Now consider the minimal prime divisor $r$ of $|W|$,
 which is nontrivial by hypothesis.%
\margin{$r$}
By Fact \ref{f:Weyl_p},
 the centralizer of any $r$-element $a$ representing an $r$-element of $W$
 has a nontrivial $r$-unipotent subgroup. 
So Lemma \ref{a_not_piT} says :

\begin{corollary}\label{r_not_piT}
$r \in \tau'$.
\end{corollary}

\noindent In particular, there are Weylian $\tau'$-elements.

In \S\ref{sec:Cyclicity} we shall prove that any element of
 the Weyl group lifts to a Weylian $\tau'$-element of $G$.

\subsection{Carter subgroups}

We next argue that the Cartan subgroup $C(T)$
 is in fact a Carter subgroup of $G$.

\begin{lemma}\label{CTW=1}
If $a$ is a Weylian $\tau'$-element for $T$, then $C_{C(T)}(a)$ is trivial.
\end{lemma}

\begin{proof}
We may assume that $a$ has prime order $p$.
Again let $B_a$ be the unique Borel containing $C^\o(a)$.
Consider some $x \in C_{C(T)}(a)^\#$.
By Lemma \ref{CUpBt_inf}, $C_{U_p(B_a)}(x)$ is infinite.
Thus $T \leq C^\o(x) \leq B_a$ by Fact \ref{jaligot}.
Here again we contradict the fact that connected solvable groups
 have trivial Wely groups (Fact \ref{selfnormsolv}).
Therefore $C_{C(T)}(a) = 1$, as desired.
\end{proof}

\begin{corollary}\label{CT=Q}
$C(T)$ is a Carter subgroup of $G$.
\end{corollary}

\begin{proof}
By Corollary \ref{r_not_piT},
 there is a Weylian $\tau'$-element $a$.
So $C_{C(T)}(a) = 1$ by Lemma \ref{CTW=1}.
It follows from Fact \ref{solvnilp} that $C(T)$ is nilpotent.
So the result follows because
 $C(T)$ is almost self-normalizing by \cite[Theorem 6.16]{BN}.
\end{proof}

\subsection{Invariance}

We finally show that the Cartan subgroup $C(T)$ is a Borel subgroup in certain cases.
Consider a Borel subgroup $B_T$ containing $C(T)$.%
\margin{$B_T$}

\begin{proposition}\label{goingup}
$B_T$ can be chosen to be $W$-invariant in either of the following cases :
\begin{itemize}
\item[(i)] if $C(T)$ is not abelian; or
\item[(ii)] if the \Prufer $q$-rank of $T$ is $\geq 3$ for prime $q$.
\end{itemize}
\end{proposition}

\begin{proof}
By \cite[Proposition 4.1ii]{Bu_Bender},
 every nilpotent subgroup of two distinct Borel subgroups is abelian.
So part (i) follows from Corollary \ref{CT=Q}.

Assume now (ii), i.e.\ suppose that $T$ has \Prufer $q$-rank $\geq 3$.
We may assume that no Borel subgroup containing $C^\o (T)$ is abelian,
 as otherwise $B_T = C(T)$ is clearly $W$-invariant.
In particular, every such Borel subgroup admits at least
 one unipotence parameter with $r > 0$
 \cite[Theorem 2.12]{Bu_signalizer,FJ_Conjugacy}
For each Borel subgroup $B_i$ containing $C(T)$,
 let $\tilde{q}_i$ be such a maximal unipotence parameter of $B_i$ with $r>0$.
Choose $B_T$ as to maximize its maximal unipotence parameter
 $\tilde{q}_T$ among all $\tilde{q}_i$.

Let $Y_T = [U_{\tilde{q}}(Z (F^\o (B_T))), B_T]$.
By \cite[Theorem 2.18]{Fre06,FJ_Conjugacy},
 $Y_T$ is a homogeneous $\tilde{q}_T$-subgroup of $B_T$.
Of course $Y_T$ is clearly characteristic in $B_T$ too.
If $Y_T = 1$, then $U_{\tilde{q}} (Z (F^\o (B_T))) \leq Z(B_T) \leq C(T)$, and
 Fact \ref{Marseilles} implies that
 $B_T$ is the only Borel subgroup containing $C(T)$.

So we may assume $Y_T \neq 1$.
Then, by Fact \ref{Cgeneration},
 there is an elementary abelian $q$-subgroup $V$ of index $q$
 in $\Omega_1(T_q) = \gen{ t \in T \mid |t| = q }$
 such that $X = C_{Y_T} (V) \neq 1$.
Now $X$ is central in $F^\o(B_T)$ and is actually normalized by $C(T)$.
As $C(T)$ contains a Carter subgroup of $B_T$,
 one has $B_T = F^\o (B_T)\cdot C(T)$, so $X$ is normal in $B_T$.

Consider a lifting $w \in N(T) \setminus C(T)$ of an element of $W$.
As $\pr_q(T) \geq 3$, there is some $t \in (V \cap V^w)^\#$ and $C^\o(t) \geq X, X^w$.
Let $B_t$ be any Borel subgroup containing $C^\o (t) \geq C(T)$.
Notice that $B_t \geq C^\o (t) \geq X, X^w$.
By maximality of $\tilde{q}_T$,
 it follows that $\tilde{q}_T$ is a maximal unipotence parameter of $B_t$.
In particular, Fact \ref{Marseilles} says $B_T = B_t = B_T^w$.
Consequently $B_T$ is $W$-invariant.
\end{proof}

\begin{proposition}\label{BTnilp}
Let $a$ be a Weylian $\tau'$-element for $T$.
If $a$ normalizes $B_T$,
 then $a$ has no fixed point in $B_T$, $B_T$ is nilpotent,
 and $C(T) = B_T$ is a Borel subgroup.
\end{proposition}

\begin{proof}
We may assume that $a$ has order $p$. First notice that $B_a \neq B_T$
 as there is no Weyl group in a connected solvable group (Fact \ref{selfnormsolv}).
So let $H := (B_T \cap B_a)^\o$.

We claim that $H$ is abelian.
Otherwise \cite[Theorems 4.3 \& 4.5 (8)]{Bu_Bender}
 say that $B_a$ is involved in a maximal non-abelian intersection,
 and $F^\o (B_a)$ is divisible.
But this contradicts $U_p (B_a) \neq 1$. 
Hence $H$ is abelian.

As there is no toral $p$-torsion in $G$,
 $H$ must be $p^\perp$ by Fact \ref{jaligot}.  
Consider the action of the $p$-element $a$ on
 the definable abelian connected $p^\perp$ group $H$.
By Fact \ref{f:CsumC}, we may write $H = K \oplus L$
 with $K = C_H(a)$ and $L = [H, a]$.
Assume towards a contradiction that $K \neq 1$.

Let $A \leq U_p (B_a)$ be a $B_a$-minimal subgroup.
We observe that $[A, a] = 1$ since $a \in U_p (B_a)$.
Hence $[L, a] = L$, $[A, a] = 1$, and $L$ normalizes $A$.
A commutator computation proves $[L, A] = 1$.

Suppose first that $[A, K] = 1$.  Then $[A, H] = 1$.
So, by Fact \ref{jaligot},
 $B_a$ is the only Borel subgroup containing $C^\o(H)$.
This proves that $N_{B_T}^\o (H) = H$, whence $H$ is a Carter subgroup of $B_T$.
As Carter subgroups of $B_T$ are Carter subgroups of $G$, 
 Lemma \ref{CTW=1} and Corollary \ref{CT=Q} imply that
 $K = 1 = C_H^\o (a)$, a contradiction.

Hence $[A, K] \neq 1$.
It follows that some nontrivial definable section of $K$ embeds
 definably into the multiplicative subgroup of a field of characteristic $p$.
So, by \cite[Corollary 9]{Wag01}, no such section can be torsion-free.
It now follows that $K$ contains a non-trivial decent torus, say $T_1$.
Notice that $a \in C(T_1)$ forces $C(T_1) \leq B_a$.
Let $T_2 \geq T_1$ be the maximal decent torus of $H$.
Then $N^\o(H) \leq N^\o(T_2) = C(T_2) \leq C(T_1) \leq B_a$.
So $H$ is a Carter subgroup of $B_T$,
 yielding the same contradiction as above.

This proves that $K = 1$.  Hence $C_{B_T}^\o (a) = K = 1$.
By Fact \ref{solvnilp}, $B_T$ is nilpotent.
So Lemma \ref{CTW=1} implies that there is no centralization at all.
\end{proof}

\smallskip

We now prove Theorem \ref{HE}.

\begin{proof}[Proof of Theorem \ref{HE}]
Let $G$ be a minimal connected simple group of degenerate type,
 having a nontrivial maximal decent torus $T$
 and a nontrivial Weyl group $W = N(T)/C(T)$.
By Corollary \ref{CT=Q},
 a Cartan subgroup $C(T)$ is a Carter subgroup of $G$.
The converse is also true because 
 Carter subgroups are conjugate in $G$ \cite{Frecon_Carter_conjugacy}.

If either $C(T)$ is abelian or $T$ has \Prufer $q$-rank $\geq 3$ for some prime number $q$,
 then, by Proposition \ref{goingup}, there is a $W$-invariant Borel subgroup containing $C(T)$.
This Borel subgroup is then nilpotent by Proposition \ref{BTnilp}, 
 and hence equal to $C(T)$.
\end{proof}

\subsection{Consequences}

\begin{corollary}\label{BW}
Let $a$ be a Weylian $\tau'$-element normalizing $T$ and
 let $B_a$ be the unique Borel subgroup containing $C^\o(a)$.
Then:
\begin{itemize}
\item for each prime $p$ dividing $|a|$, $B_a$ interprets a bad field of characteristic $p$;
\item $B_a$ contains a nontrivial decent torus;
\item If $B_a$ contains a $q$-torus of \Prufer rank 2 for some prime $q$,
      then it contains a Carter subgroup of $G$;
\item $B_a$ does not contain a $q$-torus of \Prufer rank $\geq 3$ for any prime $q$.
\end{itemize}
\end{corollary}

\begin{proof}
We may assume that $a$ has order $p$.
Fix a $B_a$-minimal subgroup $A$ of $U_p(B_a)$.
If $A \leq Z(B_a)$, then, by Fact \ref{jaligot},
 $B_a \cap B_a^g \neq 1$ implies $g \in N(B_a)$.
Hence $B_a$ is disjoint from its distinct conjugates,
 and a standard rank computation shows that $B_a^G$ is generic in $G$. 
On the other hand, a generic element $g$ of $G$ contains a maximal decent torus inside
 its definable closure $d(g)$ by \cite[Theorem 1]{BC_semisimple}.
This implies that $B_a$ contains a maximal decent torus $T^g$.
But since $B_a$ is disjoint from its distinct conjugates,
 $B_a$ is the only conjugate of $B_a$ containing $T^g$.
In particular $a^g \in N(B_a)$.

As usual $a^g \in B_a$ by Fact \ref{unipotentselfnormalization},
 whence $a \in B_a \cap B_a^{g^{-1}}$. 
So, by Fact \ref{jaligot}, $g \in N(B_a)$.
As $T \leq d(g)$, we have $T \leq B_a$.
But this contradicts $a \in B_a$ for the usual reasons (Fact \ref{selfnormsolv}).

Hence $A$ is not central in $B_a$. Zilber's field theorem \cite[Theorem 9.1]{BN} now says that $B_a$ interprets a field of characteristic $p$, which is bad since the group has degenerate type.
On the other hand, $B_a$ has a nontrivial decent torus $T_0$ because the latter field has a locally finite model by \cite[Corollary 9]{Wag01}.

Assume that $B_a$ contains a $q$-torus $T_q$ of \Prufer $q$-rank $\geq 2$, and consider the action of $T_q$ on $A$. As $T_q$ can't embed into the multiplicative group of any field, there is $t \in T_q$ such that $A \leq C^\o(t)$. By Fact \ref{jaligot}, $C(T_q) \leq C^\o (t) \leq B_a$. But $C(T_q)$ contains a Cartan subgroup of $G$, which is a Carter subgroup of $G$ by Corollary \ref{CT=Q}.

Now assume that $T_q$ has \Prufer $q$-rank $\geq 3$; say $T_q \leq T^g$. Then by Theorem \ref{HE}, $B_a$ is a Carter subgroup of $G$, so $B_a = C(T^g)$, and we argue as in the first paragraph to get a contradiction.
\end{proof}

The next corollary will play a key role in \S\ref{sec:Cyclicity}.

\begin{corollary}\label{astrWxstrW}
If $a$ is Weylian $\tau'$-element for $T$
 and $1 \neq x \in N(T) \cap C(a)$,
then $x$ is a Weylian $\tau'$-element and $B_x = B_a$.
\end{corollary}

\begin{proof}
We may assume that $a$ is a $p$-element as usual.

We first prove that $x$ is torsion and not inside $C(T)$.
Let $y = x^n$ be such that $y \in C(T)$.
Then $y \in C_{C(T)} (a) = 1$ by Lemma \ref{CTW=1}.
So it follows that $x$ is torsion and not inside $C(T)$.

We may now assume that $x$ is a $q$-element.
If $p = q$, then $x \in B_a$ by Fact \ref{BCx}, as desired.
So we may assume that $q \neq p$.
It follows from uniqueness of $B_a$ that $x$ normalizes $B_a$.
So Lemma \ref{CUpBt_inf} says that $x$ centralizes elements of $U_p(B_a)$.
Now, by Fact \ref{jaligot},
 $B_a$ is the only Borel subgroup containing $C^\o(x)$.

In view of Lemma \ref{a_not_piT},
 it suffices to show that $C^\o(x)$ 
 has a nontrivial $q$-unipotent subgroup. 
So suppose that $U_q( C^\o(x) ) = 1$.
Then $q$ is not the minimal prime divisor of $|W|$ by Fact \ref{f:Weyl_p},
 so $q \geq 5$ in particular.
Moreover $x$ is toral in $G$ by Fact \ref{f:torality},
 i.e.\ $x \in T^g$ for some $g \in G$.
So then $T^g \leq C^\o(x) \leq B_a$.

By Fact \ref{irrpol}, $T^g$ must have \Prufer $q$-rank at least $q-1 \geq 4$,
 contradicting Corollary \ref{BW}.
\end{proof}

\section{Isomorphic lifting}\label{sec:Cyclicity}

The main result of the paper is the following
 characterization of the Weyl group.

\begin{theorem}\label{thm:Weyl_cyclic}
Let $G$ be a minimal connected simple group of finite Morley rank,
 and let $T$ be a decent torus of $G$.
Then the Weyl group $W := N(T)/C(T)$ is cyclic,
 and has an isomorphic lifting to $G$;
 moreover no primes dividing $|W|$ appear in $T$ except possibly 2.
\end{theorem}

To prove this, we first observe that the decent torus $T$ may be taken maximal
 by Lemma \ref{lem:Weyl_increase}.

In fact, our proof needs only cover groups of {\em degenerate type}.
So let us first reduce to this case.
If $G$ has a nontrivial 2-unipotent subgroup, then
 $G$ is algebraic by \cite{ABC_EvenType}, and
 hence $G \cong \PSL_2$ where $|W| = 2$.
If $G$ has a non-trivial 2-torus, then $|W| = 1,2,3$
 by the comments following Lemma \ref{lem:Weyl_increase}
 and the second author's thesis work:

\begin{fact}[{\cite[see Th\'{e}or\`{e}me p.89 and p.91]{DeloroPhd}}]\label{DJ_odd_cyclic}
Let $G$ be a minimal connected simple group of finite Morley rank of odd type,
 and let $Q$ be a Carter subgroup of $G$ containing a Sylow\o 2-subgroup.
Then $W_Q := N(Q)/Q$ has order $1$, $2$, or $3$.
\end{fact}

\noindent
Moreover, if the Weyl group has order 3, then
 \cite[Corollaire 5.5.15]{DeloroPhd} says that
 $N(Q) = Q \rtimes \gen{\sigma}$ where $\sigma$ has order 3.
So by Fact \ref{f:Weyl_p}, $U_3(C^\o (\sigma)) \neq 1$,
 and Lemma \ref{a_not_piT} proves that $\sigma$ is a Weylian $\tau'$-element,
 which means that $G$ has no non-trivial $3$-torus.

So Theorem \ref{thm:Weyl_cyclic} holds in odd type.
As such, we may assume that the Sylow 2-subgroup is finite,
 and hence trivial by \cite{BBC}.

We also recall an earlier analysis, by the first author and Tuna \Altinel,
 of the Weyl group $W := N(T)/C(T)$ in a minimal connected simple group $G$.

\begin{fact}[{\cite[Proposition 5.1 \& Corollary 5.7]{AB_CT}}]
\label{metacyclic}
The Weyl group $W$ of $G$ is a metacyclic Frobenius complement
 in a finite Frobenius group.
\end{fact}

A \emph{Frobenius complement} is the stabilizer of a point in a Frobenius group;
which is a transitive permutation group,
 such that no non-trivial element fixes more than one point
 and some non-trivial element fixes a point.
Such groups have a quite restrictive structure described by the following.

\begin{fact}[{\cite[10.3.1 p.~339]{Gor}}]\label{Frobenius}  
Let $W$ be a Frobenius complement in a finite Frobenius group.  Then
\begin{enumerate}
\item Any subgroup of $W$ of order $pq$, $p$ and $q$ primes, is cyclic.
\item Sylow $p$-subgroups of $W$ are either cyclic or
 possibly generalized quaternion groups if $p=2$.
\end{enumerate}
\end{fact}


\noindent Metacyclicity follows for any Frobenius complement  
 inside a degenerate type group  \cite[7.6.2 p.~258]{Gor}.

For us, the important fact about these groups is the following.

\begin{lemma}\label{lem:metaFrobGreg}
Let $X$ be a finite group with cyclic Sylow subgroups
 such that any subgroup of order $pq$ with $p$, $q$ primes, is cyclic.
Then any two elements of prime order commute.
\end{lemma}

\noindent In particular any $x,y \in X$ have nontrivial powers which commute.

\begin{proof}
Since the Sylow 2-subgroups of $X$ are cyclic,
 we know that $X$ is solvable by \cite[7.6.2 p.~258]{Gor}.
Also $F(X)$ is abelian because its Sylow 2-subgroups are cyclic.
Now if $F(X) \leq Z(X)$ then $X = F(X)$ is abelian, and we are done.

So assume $F(X) \not \leq Z(X)$.
Choose $p$ prime such that
 the Sylow $p$-subgroup $P$ of $F(X)$ is not central in $X$.
Then $X$ induces a nontrivial $p^\perp$-action on $P \simeq \Z/p^n\Z$, by hypothesis.
It follows that $X$ induces a nontrivial action on $\Omega_1(P)$.
So let $a \in \Omega_1(P) \setminus Z(X)$. Then $C(a) < X$.

Now $\gen{a} \vartriangleleft X$ so for any element $x \in X$ of prime order,
 $\gen{a, x}$ is cyclic by hypothesis.
Thus, if $x, y \in G$ have prime order, then $x, y \in C(a)$, and we conclude by induction.
\end{proof}

\smallskip

With the above reductions \& lemma at hand, we turn to proving Theorem \ref{thm:Weyl_cyclic},
 by considering the following situation.

\begin{hypothesis}
Our minimal connected simple group $G$ of degenerate type
 has a nontrivial maximal decent torus $T$ with
 a Weyl group $W := N(T)/C(T)$ that is a metacyclic Frobenius complement.
\end{hypothesis}

Again $\tau'$ is the set of primes for which $G$,
 or equivalently $T$, has no $p$-torus.
Let $|W| = \Pi_{i=1}^k p_i^{n_i}$ with $p_1 < \dots < p_k$.
We shall construct a sequence $(a_i)_{i = 1 \dots k}$
 of commuting Weylian $\tau'$-elements such that
 $w := \Pi_{i=1}^k a_i$ generates $W$ modulo $C(T)$.
To restate this, there are $a_i$ for $i = 1 \dots k$ such that
\renewcommand{\theenumi}{(\roman{enumi})}
\begin{enumerate}
\item  each $a_i \in N(T) \setminus C(T)$ has order $p_i^{n_i}$,
\item  all the $a_i$'s commute, and
\item  $G$ has no $p_i$-torus.
\end{enumerate}
A priori, it doesn't matter in what order these $a_i$ appear,
 but $p_1$ is for free, thanks to Fact \ref{f:Weyl_p} via Corollary \ref{r_not_piT}.

Let us start.
By Fact \ref{Frobenius} (ii), $W$ contains an element of order $p_1^{n_1}$.
We lift this element into a $p_1$-element $a_1$ of $N(T)$.
Lemma \ref{lem:constant_order} says that $a_1$ has the adequate order,
 and Corollary \ref{r_not_piT} explains why $p_1 \in \tau'$.
So $a_1$ is a Weylian $\tau'$-element.
Set $B_W := B_{a_1}$.

\smallskip

We now perform the induction.
Assume that $a_1, \dots, a_i$ have been constructed as claimed.
Let $b_i = a_1 \dots a_i$.
So then $b_i$ is a Weylian $\tau'$-element.
By Corollary \ref{astrWxstrW},
 $B_{a_1} = \dots = B_{a_i} = B_{b_i}$.
Also for convenience let $\pi_i = \{p_1, \dots, p_i\}$,
 and $\beta_i$ denote the image of $b_i$ in $W$.

\begin{claim}\label{claima'}
There is a $p_{i+1}$-element $a'$ of $N(T) \setminus C(T)$
 that commutes with a power $b'$ of $b_i$.
\end{claim}

\begin{proof}
By Lemma \ref{lem:metaFrobGreg}, there is a $p_{i+1}$-element $\alpha'$ of $W$
 commuting with a non-trivial power $\beta'$ of $\beta_i$.
Say $\beta' = \beta_i^k$, and let $b' = b_i^k \neq 1$.
Let $a_0$ be a lifting of $\alpha'$ into
 a $p_{i+1}$-element of $N(T) \setminus C(T)$.
So far $[a_0, b']$ is in $C(T)$ but need not equal the identity.
Therefore we shall adjust $a_0$ to get a suitable $a'$.

By Lemma \ref{UpCT=1}, $C(T)$ is $\pi_i^\perp$.
It follows that $\gen{b'}$ is a Hall $\pi_i$-subgroup of
 $K = C(T) \cdot \gen {b'} = C(T) \rtimes \gen {b'}$.
As $K$ has $\pi_i^\perp$-type, it conjugates its Sylow $\pi_i$-subgroups
 by \cite[Theorem 4]{BC_semisimple};
 as a straightforward consequence, $K$ conjugates its Hall $\pi_i$-subgroups.

Now $a'$ normalizes $K$. Let $H = K \cdot \gen{a_0}$.
A Frattini argument says that $H = K \cdot N_H (\gen{b'})$,
 so there must be a $p_{i+1}$-element $a'$ in $N_H (\gen{b'})$.
It follows that $a'$ does not lie inside $C(T)$.
Since $p_{i+1} > \max \pi_i$, $|b'|$ is a $\pi_i$-number,
 and $\gen{b'}$ has no automorphisms of order $p_{i+1}$.
Therefore $a'$ must centralize $b'$.
\end{proof}

\begin{claim}
There is a Weylian $\tau'$-element $a_{i+1}$ of order $p_{i+1}^{n_{i+1}}$ such that $B_{a_{i+1}} = B_W$.
\end{claim}

\begin{proof}
Let $a'$ and $b'$ be as in Claim \ref{claima'}. Notice that $b'$, being a power of $b_i$, is a Weylian $\tau'$-element and $B_{b'} = B_W$. By Corollary \ref{astrWxstrW}, $a'$ is a Weylian $\tau'$-element and $B_{a'} = B_W$.

By Fact \ref{Frobenius} (2), Sylow $p_{i+1}$-subgroups of $W$ are cyclic. Let $\alpha_{i+1}$ be any element of $W$ of order $p_{i+1}^{n_{i+1}}$ generating a Sylow $p_{i+1}$-subgroup containing the image of $a'$. Let $a''$ be a lifting of $\alpha_{i+1}$ into a $p_{i+1}$-element of $N(T) \setminus C(T)$. By Lemma \ref{lem:constant_order}, $a''$ has order $p_{i+1}^{n_{i+1}}$. Also, notice that $[a', a''] \in C(T)$; we'll use the same argument as in Claim \ref{claima'}.

Let $K = C(T) \rtimes \gen{a'}$ and $H = K \cdot \gen{a''} \leq N(K)$. Since $C(T)$ is $p_{i+1}^\perp$, $\gen{a'}$ is a Hall $p_{i+1}$-subgroup of $K$. There must therefore be a $p_{i+1}$-element of $N_{N(T)} (\gen{a'})$ lifting $\alpha_{i+1}$.
Let $a_{i+1}$ be this element; it has order $p_{i+1}^{n_{i+1}}$, it is a Weylian $\tau'$-element; and above all, $a_{i+1}$ normalizes $\gen{a'}$, whence $B_{a_{i+1}} = B_{a'} = B_W$.
\end{proof}

Now $a_{i+1}$ commutes with $b_i$ because $B_{a_{i+1}} = B_W = B_{b_i}$ and
 $a_{i+1}$ and $b_i$ are unipotent elements of $B_W$ of coprime order.
So the induction is completed, and eventually we set $w := b_k$.
Then $w$ has order $|W|$ both in $G$ and in $W$;
 in particular $W \simeq \gen{w}$ is cyclic.
Besides, $w$ is a Weylian $\tau'$-element and $B_w = B_W$.
By Corollary \ref{UpCT=1}, $C(T)$ is $\pi_k^\perp$,
 so no prime number occurring in $T$ can divide $|W|$.

This concludes the proof of Theorem \ref{thm:Weyl_cyclic}.  \qed

\medskip

Connectedness of the Sylow $p$-subgroups is an immediate consequence.

%

\begin{corollary}
Let $G$ be a minimal connected simple group and $p \neq 2$ a prime.
Then the maximal $p$-subgroups of $G$ are connected.
\end{corollary}

\begin{proof}
Let $S$ be a maximal $p$-subgroup of $G$.
If $d(S) = G$, then $S$ is connected.

So we may assume $d(S) < G$. In this case, $S^\o \leq d^\o (S)$ which is solvable,
 and as usual $S^\o = T \ast U$ with $T$ a $p$-torus and $U$ $p$-unipotent.
Also $S^\o \neq 1$ by \cite{BBC}.

Suppose first that $U \neq 1$.
Then there is a unique Borel subgroup $B$ containing $U$.
If $s \in S$, Fact \ref{jaligot} forces $B^s = B$, that is $s \in N(B)$.
But $s \in B$ by Fact \ref{unipotentselfnormalization}.
As $S \leq B$ is a Sylow $p$-subgroup of the connected solvable group $B$,
 it is connected.

Suppose otherwise that $U = 1$.
So $S^\o = T$ and $S$ is toral-by-finite.
Here the result follows by Theorem \ref{thm:Weyl_cyclic}.
\end{proof}

\section{Counting Borel subgroups}\label{sec:Counting_Borel}

We conclude by noticing that only finitely many Borel subgroups
 contain the Carter subgroup of a minimal connected simple group
 with a ``large" decent torus.
If there is a Weyl group $W$, this number is either
 a multiple of $|W|$, or simply 1.

\begin{lemma}
Let $G$ be any group of finite Morley rank.
Fix a Carter subgroup $Q$ of $G$ and a Borel subgroup $B$ of $G$ containing $Q$. 
Then there are finitely many $G$-conjugates of $B$ containing $Q$.
\end{lemma}

\begin{proof}
Consider the set $\mathcal{S} := \{ (x, y) : x \in G/N_G(Q), y \in G/N_G (B), Q^x \leq B^y\}$.

The first projection has fibers $\mathcal{S}_x = \{ y \in G/N_G (B): Q^x \leq B^y\}$. These sets are uniformly definable and have the same rank, say $k$.

The second projection has fibers $\mathcal{S}_y = \{ x \in G/N_G (Q): Q^x \leq B^y\}$, which are uniformly definable.  Since solvable groups of finite Morley rank conjugate their Carter subgroups, $\mathcal{S}_y$ has rank $\rk B - \rk N_B (Q)$.

Putting this together, we find
$$\rk \mathcal{S} = \rk G - \rk N_G (Q) + k = \rk G - \rk N_G (B) + \rk B - \rk N_B (Q),$$
whence $k = 0$.
\end{proof}

We exploit the theory of maximal intersections of Borel subgroups \cite{Bu_Bender}
 to bound the number of conjugacy classes (see also \cite[Proposition 5.5.3]{DeloroPhd}).

\begin{lemma}
Let $G$ be a minimal simple connected group of finite Morley rank.
Suppose that for some prime number $q$,
 there is a $q$-torus of \Prufer-rank at least two.
Then a Carter subgroup of $G$ is included in
 finitely many non-conjugate Borel subgroups.
\end{lemma}

\begin{proof}
Let $T_q$ be a $q$-torus of \Prufer $q$-rank $d \geq 2$, and
let $k = 1 + \dots + q^{d-1}$ be the number of lines
 in $\Omega_1(T_q) := \gen{ t \in T_q \mid |t| = q }$.
Let $t_1, \dots, t_{k}$ be vectors of $\Omega_1(T_q)$ representing these lines.
For $\ell = 1, \dots, k$, fix some Borel subgroup $\beta_\ell \geq C^\o(t_\ell)$.
Also let $\tilde{q}_\ell$ be maximal unipotence parameters for $\beta_\ell$
 ($\ell = 1, \dots, k$).

Suppose towards a contradiction that there is an infinite set $I$ with
 pairwise non-conjugate Borel subgroups $B_i$ for $i \in I$,
 all of which contain $Q$.
Since the $B_i$'s are pairwise non-conjugate,
 we may assume that none of the $B_i$ is conjugate to any of the $\beta_\ell$.
For each $i$, choose a maximal unipotence parameter $\tilde{p}_i$ for $B_i$.

Fix some $i$.  Then $B_i \geq Q \geq T_q$.
So Fact \ref{Cgeneration} implies
 $B_i = \gen{ C_{B_i}^\o (t_\ell) : \ell = 1, \dots, k }
 = \gen{ (B_i \cap \beta_\ell)^\o : \ell = 1, \dots, k }$.
It follows that at least two of the intersections
 $H_{i, \ell} = (B_i \cap \beta_\ell)^\o$ are distinct and non-abelian.

So we may assume that for all $i \in I$,
 $H_{i, 1}$ and $H_{i, 2}$ are distinct and non-abelian.
Since $B_i$ is never conjugate to $\beta_1$ nor to $\beta_2$,
 every $H_{i, 1}$ and every $H_{i, 2}$ is a maximal non-abelian intersection
 \cite[Theorem 4.3]{Bu_Bender}.
The same result tells us that one cannot have
 $\tilde{p}_i = \tilde{q}_1$ nor $\tilde{p}_i = \tilde{q}_2$.

If there are indices $i$ and $j$ such that
 $\tilde{q}_1 > \tilde{p}_i$ and $\tilde{q}_1 > \tilde{p}_j$,
then \cite[Lemma 3.30]{Bu_Bender} says that $B_i$ is conjugate to $B_j$,
 a contradiction.
So infinitely often, it must be the case that $\tilde{p}_i > \tilde{q}_1$.

Similarly we may assume that $\tilde{p}_i > \tilde{q}_2$ for all $i \in I$.
So \cite[Lemma 3.30]{Bu_Bender} says that $\beta_1$ is conjugate to $\beta_2$.
Since $F_s (\beta_\ell)$ is not abelian only when $s = \rr_0(H_{i, \ell}')$
 by \cite[Theorem 4.5 (4)]{Bu_Bender}, we find that
$r' = \rr_0 (H_{i, \ell}')$ does not depend on $i \in I$ nor on $\ell = 1, 2$.

Set $Q_{r'} := U_{(0, r')} (Q)$.
Then, by \cite[Lemma 3.23]{Bu_Bender}, $Q_{r'}$ is non-trivial and
 central in $H_{i, \ell}$, for all $i \in I$ and $\ell = 1, 2$.
Now $ \gen{ H_{i, 1} , H_{i, 2} } \leq C^\o (Q_{r'})$.

Recall that $B_i$ and $\beta_1$ are the only Borel subgroups containing $H_{i, 1}$
 and that $B_i$ and $\beta_2$ are the only Borel subgroups containing $H_{i, 2}$.
As $H_{i, 1} \neq H_{i, 2}$, it follows that $\beta_1 \neq \beta_2$.
Then $B_i$ is the only Borel subgroup containing $C^\o (Q_{r'})$.
In particular, $B_j = B_i$ for any $j \in I$, a contradiction.
\end{proof}

We note that an effective version of this result is possible, but
 the obvious bound does not seem useful at present.

\begin{corollary}
Let $G$ be a minimal connected simple group of finite Morley rank.
Suppose there is a $q$-torus of \Prufer $q$-rank $\geq 2$ for some prime number $q$,
 and a non-trivial Weyl group.
Then the number of Borel subgroups containing a fixed Carter subgroup of $G$
 is a multiple of $|W|$ or possibly just 1.
\end{corollary}

\begin{proof}
This number is finite by the two preceding lemmas.
By Theorem \ref{thm:Weyl_cyclic},
 we lift $W$ to an isomorphic subgroup $\hat{W} \leq G$, and
 let $\hat{W}$ act on a $G$-conjugacy class of Borel subgroups containing $Q$.
By a Frattini argument, $\hat{W}$ acts transitively on
 each $G$-conjugacy class of such Borel subgroups.
If some $a \in \hat{W}$ fixes a Borel subgroup $B \geq Q$,
 then $B = Q$ by Proposition \ref{BTnilp}.
So unless $Q$ is a Borel subgroup,
 $\hat{W}$ acts freely on each conjugacy class.
\end{proof}

\small
\bibliographystyle{alpha} 
\bibliography{fMr,burdges-articles,burdges-preprints,burdges-books}

\end{document}